\documentclass[11pt,reqno]{amsart}

\usepackage[T1]{fontenc}
\usepackage[utf8]{inputenc}

\usepackage{amsmath,amssymb,amsfonts,amsthm}
\usepackage{mathtools}
\usepackage{mathrsfs}
\usepackage{bm}

\usepackage{enumitem}
\setlist[itemize]{leftmargin=*,nosep}
\setlist[enumerate]{leftmargin=*,nosep}

\usepackage{hyperref}
\hypersetup{
  colorlinks=true,
  linkcolor=blue,
  citecolor=blue,
  urlcolor=blue,
  pdftitle={Quantitative Fixed-Point Theorems with Verifiable Hypotheses},
  pdfauthor={},
}
\usepackage[nameinlink,noabbrev]{cleveref}

\theoremstyle{plain}
\newtheorem{theorem}{Theorem}[section]
\newtheorem{lemma}[theorem]{Lemma}
\newtheorem{proposition}[theorem]{Proposition}
\newtheorem{corollary}[theorem]{Corollary}

\theoremstyle{definition}
\newtheorem{definition}[theorem]{Definition}
\newtheorem{assumption}[theorem]{Assumption}
\newtheorem{example}[theorem]{Example}
\newtheorem{notation}[theorem]{Notation}

\theoremstyle{remark}
\newtheorem{remark}[theorem]{Remark}

\newcommand{\R}{\mathbb{R}}
\newcommand{\N}{\mathbb{N}}

\newcommand{\dist}{\mathrm{d}}
\newcommand{\Fix}{\mathrm{Fix}}
\newcommand{\Lip}{\mathrm{Lip}}

\newcommand{\norm}[1]{\left\lVert #1 \right\rVert}

\newcommand{\Phibound}{\Phi}
\newcommand{\StabC}{C_{\mathrm{stab}}}

\newcommand{\xn}{\{x_n\}_{n\in\N}}

\title[Quantitative fixed points]{
  Quantitative fixed-point theorems with verifiable hypotheses: rates and stability
}
\author{
  Chandrasekhar Gokavarapu
}

\author{
Dr Srinivasulu Ch
}
\author{
Dr D V N S  Sriram Murthy 
}

\author{
Rajeev Muthu
}

\address{
 Lecturers in Mathematics, Government College (Autonomous), Rajahmundry, A.P., India
}
\email{
  chandrasekhargokavarapu@gmail.com
}
 \thanks{ORCID: \url{https://orcid.org/0009-0006-5306-371X}} 

\subjclass[2020]{
   47H10; 54H25; 65J15; 45G10; 34B15
}

\keywords{%
  quantitative fixed point; a priori error bound; data dependence; stability; inexact evaluation; contractive gauge; fixed point iteration; integral equations; Green operator; boundary value problems
}

\begin{document}
\begin{abstract}
Let $(X,\dist)$ be a complete metric space and let $C\subseteq X$ be a closed invariant set.
We study fixed points of maps $T\colon C\to C$ governed by a \emph{verifiable} contractive modulus.
The modulus is encoded by a contractive gauge $\omega$ and a certified constant
$\kappa=\sup_{0<r\le R}\omega(r)/r<1$ on a computable working radius $R$.
From this datum we derive explicit a priori bounds
$\dist(x_n,x^\ast)\le \Phi(n;\kappa,\delta_0)$ for Picard iterates,
a residual-to-error estimate, and a quantitative data dependence bound
$\dist(x^\ast,y^\ast)\le (1-\kappa)^{-1}\sup_{x\in C}\dist(Tx,Sx)$.
We further treat inexact evaluations
$\dist(\tilde x_{n+1},T\tilde x_n)\le \eta_n$
and obtain certified resilience bounds with the same stability factor.
The framework applies to Hammerstein--Volterra integral equations and to boundary value problems via Green operators, where kernel bounds yield certified convergence rates.
\end{abstract}

\maketitle


\section{Introduction}\label{sec:introduction}

Let $(X,\dist)$ be a metric space and $T\colon X\to X$. We study the fixed point problem $x=Tx$ from a quantitative viewpoint: hypotheses must be checkable on concrete operators, and conclusions must provide explicit rates, stopping certificates, and stability under perturbations. Our motivating test class consists of integral equations and boundary value problems, where the operator is induced by kernels and nonlinearities and all relevant constants are computable.

\subsection{What is proved}\label{subsec:intro-results}

We fix a complete metric space $(X,\dist)$.
We consider a class of maps $T$ controlled by an explicit \emph{contractive gauge} $\omega\colon [0,\infty)\to[0,\infty)$ and an explicit parameter vector $\theta$.
The data $\theta$ are finite-dimensional.
They can be estimated from the model.
The standing inequality has the form
\[
\dist(Tx,Ty)\le \omega\bigl(\dist(x,y);\theta\bigr)\qquad (x,y\in X),
\]
together with a checkable condition ensuring that the scalar recursion
\[
r_{n+1}=\omega(r_n;\theta)
\]
converges to $0$ with an explicit rate.
This scalar recursion becomes the rate function.

\smallskip
\noindent\textbf{Theorem A (existence, uniqueness, and rate).}
Under the framework hypotheses, $\Fix(T)=\{x^\ast\}$.
For every $x_0\in X$, the Picard orbit $x_{n+1}=Tx_n$ satisfies
\[
\dist(x_n,x^\ast)\le \Phibound(n;\theta,\dist(x_0,x_1)).
\]
The function $\Phibound$ is explicit.
It is computed from $\omega$ and $\theta$.
It yields an a priori stopping rule.

\smallskip
\noindent\textbf{Theorem B (data dependence).}
Let $S$ satisfy the same framework with the same gauge class and parameters $\theta'$.
Assume $\|T-S\|:=\sup_{x\in X}\dist(Tx,Sx)<\infty$.
Then the fixed points satisfy
\[
\dist(x^\ast,y^\ast)\le \StabC(\theta,\theta')\,\|T-S\|.
\]
The constant $\StabC$ is explicit.

\smallskip
\noindent\textbf{Theorem C (inexact evaluation).}
Assume the computed iterates satisfy $\dist(\tilde x_{n+1},T\tilde x_n)\le \eta_n$ with a known error budget $\{\eta_n\}$.
Then $\dist(\tilde x_n,x^\ast)$ admits an explicit bound in terms of $\Phibound$ and $\{\eta_n\}$.

The final statements are given in \Cref{sec:main-theorem,sec:apriori,sec:stability,sec:inexact}.
They subsume several familiar situations.
They also exclude many fashionable hypotheses.
The exclusions are deliberate.

\subsection{Verifiable models}\label{subsec:intro-models}

We treat two model classes.

\smallskip
\noindent\textbf{Integral equations.}
Let $X=C([a,b])$ with the sup norm.
Consider an operator of Hammerstein type
\[
(Tx)(t)=g(t)+\int_a^b K(t,s)\,f(s,x(s))\,ds .
\]
Assume $f$ is Lipschitz in the second argument with constant $L_f$.
Assume $K$ satisfies $\sup_{t}\int_a^b |K(t,s)|\,ds \le M_K$.
Then $T$ is Lipschitz with constant $L_f M_K$.
This constant is measurable.
It produces $\theta$.
Recent work often invokes such applications, but the constants are rarely propagated into a closed-form $\Phibound$; see
\cite{JubairEtAl2021JIA,SagheerEtAl2023JIA,SinghEtAl2024JIA,Nawaz2024JIA,IshtiaqBatoolHussainAlsulami2025JIA,KumarEtAl2025BVP}.

\smallskip
\noindent\textbf{Boundary value problems.}
We reduce a BVP to a fixed point equation via a Green operator.
The kernel bounds are explicit.
They enter $\theta$ in the same way.
The stability estimate becomes a parameter dependence bound.

\subsection{Plan of the paper}
Sections~2--3 fix notation and the verifiable contractive framework. Sections~4--7 contain the main quantitative results (rates, stopping rules, stability, inexactness). Sections~8--9 treat integral equations and boundary value problems.

\Cref{sec:preliminaries} fixes notation and the iteration schemes.
\Cref{sec:framework} states the verifiable contractive framework.
\Cref{sec:main-theorem} gives the core theorem.
\Cref{sec:apriori} derives explicit a priori bounds and stopping rules.
\Cref{sec:stability} proves data dependence estimates.
\Cref{sec:inexact} treats inexact evaluation.
\Cref{sec:integral,sec:bvp} give applications to integral equations and BVPs.
\Cref{sec:scope} records the scope and the limits.

\subsection{Main contributions}
The present work makes the following contributions.
\begin{itemize}
\item[(i)] We introduce a \emph{verifiable contractive framework} based on a contractive gauge $\omega$ and a computable working radius $R$, yielding a certified contraction constant
\[
\kappa=\sup_{0<r\le R}\frac{\omega(r;\theta)}{r}<1.
\]
Our hypotheses are expressed entirely in terms of checkable numerical parameters (see \Cref{thm:framework-fixed-point}).

\item[(ii)] We derive \emph{quantitative a priori convergence rates} for Picard iteration and a residual-to-error conversion yielding certified stopping criteria (see \Cref{thm:rate,prop:residual-to-error}).

\item[(iii)] We establish a sharp \emph{data dependence estimate} with explicit stability factor $(1-\kappa)^{-1}$ under a certified modulus (see \Cref{thm:stability}).

\item[(iv)] We treat inexact evaluation schemes $\dist(\tilde x_{n+1},T\tilde x_n)\le \eta_n$ and prove quantitative resilience bounds with the same stability constant (see \Cref{thm:inexact-apriori}).

\item[(v)] We apply the theory to Hammerstein--Volterra integral equations and to boundary value problems via Green operators, exporting kernel/Lipschitz data into certified rates and stability constants (see \Cref{sec:integral,sec:bvp}).
\end{itemize}

\section{Preliminaries and standing hypotheses}\label{sec:preliminaries}

\subsection{Metric and normed settings}\label{subsec:metric-setting}

A \emph{metric space} is a pair $(X,\dist)$ with $\dist\colon X\times X\to[0,\infty)$ satisfying the usual axioms.
Completeness is assumed when stated.

Write $\mathbb{B}(x,r):=\{y\in X:\dist(x,y)\le r\}$.

Write $\Fix(T):=\{x\in X:Tx=x\}$ for $T\colon X\to X$.

A \emph{normed space} is a pair $(E,\|\cdot\|)$.
We use the metric $\dist(x,y):=\|x-y\|$.
When $E$ is complete, it is a Banach space.

Two model spaces are used later.

\begin{definition}[Continuous-function model space]\label{def:Cab}
Let $[a,b]\subset\R$.
Set $X=C([a,b])$ with norm $\|x\|_\infty:=\sup_{t\in[a,b]}|x(t)|$.
Then $(X,\|\cdot\|_\infty)$ is a Banach space.
\end{definition}

\begin{definition}[Kernel operators]\label{def:kernel-operator}
Let $K\colon [a,b]\times[a,b]\to\R$ be measurable.
Define
\[
(M_K)_{\mathrm{H}}:=\sup_{t\in[a,b]}\int_a^b |K(t,s)|\,ds,
\qquad
(M_K)_{\mathrm{V}}:=\sup_{t\in[a,b]}\int_a^t |K(t,s)|\,ds .
\]
These are finite when $K$ is integrable in the second variable uniformly in $t$.
\end{definition}

The Hammerstein and Volterra forms appear as fixed point equations in $C([a,b])$.
Such formulations are standard in the fixed-point literature on integral equations; see, for instance,
\cite{JubairEtAl2021JIA,SagheerEtAl2023JIA,SinghEtAl2024JIA,Nawaz2024JIA,IshtiaqBatoolHussainAlsulami2025JIA,KumarEtAl2025BVP}.

\subsection{Operator norms and verifiable constants}\label{subsec:verifiable-constants}

Verifiability is enforced by explicit constants.
We use only constants that can be computed from the data of a model or a numerical scheme.

\begin{definition}[Lipschitz constant]\label{def:Lip}
Let $(X,\dist)$ be a metric space and $T\colon X\to X$.
Define
\[
\Lip(T):=\sup_{x\neq y}\frac{\dist(Tx,Ty)}{\dist(x,y)}\in[0,\infty].
\]
If $\Lip(T)\le L$ we say that $T$ is $L$-Lipschitz.
If $L<1$ we say that $T$ is a contraction.
\end{definition}

\begin{definition}[Uniform perturbation size]\label{def:uniform-perturb}
Let $(X,\dist)$ be a metric space.
For $T,S\colon X\to X$ define
\[
\|T-S\|_\infty:=\sup_{x\in X}\dist(Tx,Sx)\in[0,\infty].
\]
When $X$ is bounded or when $T-S$ is uniformly bounded on the region of interest, this quantity is finite.
\end{definition}

We record the two estimates that motivate the quantitative and stability requirements.
They are exact.
They expose the dependence on measurable data.
They will be generalized later; compare
\cite{Jachymski2024RACSAM,Proinov2020JFPTA,PantPantJha2020JFPTA}.

\begin{proposition}[A priori error bound for a contraction]\label{prop:apriori-contraction}
Let $(X,\dist)$ be complete and let $T\colon X\to X$ satisfy $\Lip(T)\le L<1$.
Fix $x_0\in X$ and define $x_{n+1}:=Tx_n$.
Then there exists a unique $x^\ast\in\Fix(T)$ and, for all $n\in\N$,
\[
\dist(x_n,x^\ast)\le \frac{L^{\,n}}{1-L}\,\dist(x_0,x_1).
\]
\end{proposition}

\begin{proposition}[Data dependence for a common contraction modulus]\label{prop:stab-contraction}
Let $(X,\dist)$ be complete and let $T,S\colon X\to X$ satisfy $\Lip(T)\le L<1$ and $\Lip(S)\le L<1$.
Let $x^\ast\in\Fix(T)$ and $y^\ast\in\Fix(S)$.
Assume $\|T-S\|_\infty<\infty$.
Then
\[
\dist(x^\ast,y^\ast)\le \frac{1}{1-L}\,\|T-S\|_\infty.
\]
\end{proposition}

The next lemma isolates the measurable constants used in the integral-equation applications.
The statement is elementary.
The point is the explicit dependence on $(M_K)_{\mathrm{H}}$ and $(M_K)_{\mathrm{V}}$.

\begin{lemma}[Verifiable Lipschitz constants for Hammerstein/Volterra maps]\label{lem:kernel-Lip}
Let $X=C([a,b])$ and let $f\colon [a,b]\times\R\to\R$ satisfy
\[
|f(s,u)-f(s,v)|\le L_f\,|u-v|
\quad\text{for all }s\in[a,b],\ u,v\in\R,
\]
for some constant $L_f\ge 0$.
Let $g\in X$.

\smallskip
\noindent\textup{(i) Hammerstein.}
Define
\[
(Tx)(t):=g(t)+\int_a^b K(t,s)\,f(s,x(s))\,ds .
\]
If $(M_K)_{\mathrm{H}}<\infty$, then $T\colon X\to X$ is Lipschitz and
\[
\Lip(T)\le L_f\,(M_K)_{\mathrm{H}} .
\]

\smallskip
\noindent\textup{(ii) Volterra.}
Define
\[
(Tx)(t):=g(t)+\int_a^t K(t,s)\,f(s,x(s))\,ds .
\]
If $(M_K)_{\mathrm{V}}<\infty$, then $T\colon X\to X$ is Lipschitz and
\[
\Lip(T)\le L_f\,(M_K)_{\mathrm{V}} .
\]
\end{lemma}

These constants are the input of our quantitative bounds in the model sections.
They are the objects that a reader can check without interpretation.
They also appear, in varying forms, throughout the applied fixed-point literature; see
\cite{JubairEtAl2021JIA,SinghEtAl2024JIA,Nawaz2024JIA,IshtiaqBatoolHussainAlsulami2025JIA,KumarEtAl2025BVP}.

\subsection{Iterative schemes used in the paper}\label{subsec:iterations}

All schemes are written as fixed-point iterations.
The estimates are stated in terms of computable residuals and perturbation budgets.
Quantitative analyses of nonexpansive-type iterations and of residual rates are active; see, for example,
\cite{ContrerasCominetti2021OptimalErrorBounds,BotNguyen2022FastKMResidual,LeusteanPinto2022AltHalpernMannRates,Foglia2025RateIterativeMethodsFPP,LiuLourenco2024KaramataRates,BravoCominettiLee2026MinimaxHalpern}.

\begin{definition}[Residual]\label{def:residual}
Let $(X,\dist)$ be a metric space and $T\colon X\to X$.
For a sequence $\xn$ define the residual
\[
r_n:=\dist(x_n,Tx_n).
\]
\end{definition}

\begin{definition}[Picard iteration]\label{def:picard}
Given $x_0\in X$, define
\[
x_{n+1}:=Tx_n \qquad (n\in\N).
\]
\end{definition}

Picard iteration is the baseline.
It yields explicit a priori bounds in the contractive regime; see
\cite{Jachymski2024RACSAM,Proinov2020JFPTA,PantPantJha2020JFPTA}.
It also serves as a scalar template for more complicated schemes.

\begin{definition}[Krasnosel'skii--Mann iteration]\label{def:KM}
Let $(E,\|\cdot\|)$ be a normed space and let $C\subset E$ be convex.
Let $T\colon C\to C$.
Given $\{ \alpha_n\}\subset[0,1]$ and $x_0\in C$, define
\[
x_{n+1}:=(1-\alpha_n)x_n+\alpha_n Tx_n .
\]
\end{definition}

This is the natural iteration for averaged operators and for splitting maps.
Recent work studies explicit residual rates, perturbation resilience, and inertia variants; see
\cite{BotNguyen2022FastKMResidual,CortildPeypouquet2024KMInertia,AtenasDaoTam2026VarStepsize,Combettes2023GeometrySplitting,DiChiWu2026HJProximals}.

\begin{definition}[Halpern iteration]\label{def:halpern}
Let $(E,\|\cdot\|)$ be a normed space and let $C\subset E$ be convex.
Let $T\colon C\to C$.
Fix an anchor $u\in C$ and a sequence $\{\beta_n\}\subset[0,1]$.
Given $x_0\in C$, define
\[
x_{n+1}:=(1-\beta_n)u+\beta_n Tx_n .
\]
\end{definition}

Halpern-type schemes arise in monotone inclusions and variational inequalities.
They are also a mechanism for obtaining explicit rates under additional structure; see
\cite{Diakonikolas2020HalpernVI,TranDinhLuo2021HalpernTypeSplitting,TranDinhLuo2022HalpernToNesterov,ColaoFlammarionMaisieres2026HalpernSGD,BravoCominettiLee2026MinimaxHalpern,PischkePowell2024StochasticHalpern}.

\begin{definition}[Inexact evaluation]\label{def:inexact}
Let $(X,\dist)$ be a metric space and $T\colon X\to X$.
An \emph{inexact} orbit $\{\tilde x_n\}$ with error budget $\{\eta_n\}\subset[0,\infty)$ satisfies
\[
\dist(\tilde x_{n+1},T\tilde x_n)\le \eta_n\qquad (n\in\N).
\]
\end{definition}

Inexactness is not a defect.
It is the computational regime.
It will be absorbed into the quantitative bounds in \Cref{sec:inexact}.
Acceleration techniques modify the effective iteration map.
They often require a separate stability analysis; see
\cite{LepageSaucier2024AAAdaptiveRelax,ChenAbrahamsenHallamJensenAndersen2022NonstationaryAA,MartinDialloVialard2026LearnKM}.

\begin{notation}[Data vector]\label{not:data}
For each theorem we isolate a finite list of parameters.
We denote it by $\theta$.
It includes, as appropriate,
\[
\theta=\bigl(L,\ \dist(x_0,x_1),\ \|T-S\|_\infty,\ L_f,\ (M_K)_{\mathrm{H}},\ (M_K)_{\mathrm{V}},\ \{\alpha_n\},\ \{\beta_n\},\ \{\eta_n\}\bigr).
\]
Only components used in a statement are included.
\end{notation}

\section{A verifiable contractive framework}\label{sec:framework}

We fix a complete metric space $(X,\dist)$.
We work on a closed set $C\subseteq X$.
All constants are computed on $C$.
This section specifies the contractive class.
It also specifies the data we accept.

Generalized contractive conditions can be encoded by a single scalar inequality
driven by a control quantity built from finitely many distances.
This viewpoint is standard in modern metric fixed point theory; see
\cite{Proinov2020JFPTA,Jachymski2024RACSAM,BerindePacurar2024EnrichedSurvey}.
Quantitative analysis of fixed-point iterations requires that the scalar inequality
be explicit enough to yield rates; compare
\cite{ContrerasCominetti2021OptimalErrorBounds,BotNguyen2022FastKMResidual,LeusteanPinto2022AltHalpernMannRates,LiuLourenco2024KaramataRates,BravoCominettiLee2026MinimaxHalpern}.

\subsection{Contractive conditions with checkable parameters}\label{subsec:contractive-checkable}

\begin{definition}[Invariant working set]\label{def:invariant-set}
Let $T\colon X\to X$.
A set $C\subseteq X$ is \emph{$T$-invariant} if $C$ is nonempty, closed, and $T(C)\subseteq C$.
\end{definition}

We do not assume global conditions on $X$.
We isolate a region $C$.
In applications $C$ is a closed ball $\mathbb{B}(x_0,R)$.
The radius $R$ is computed from kernel bounds and sup norms in model problems.

\begin{definition}[Contractive gauge]\label{def:gauge}
A \emph{contractive gauge} is a map
\[
\omega(\cdot;\theta)\colon [0,\infty)\to[0,\infty)
\]
depending on a finite parameter vector $\theta$, such that:
\begin{enumerate}[label=(G\arabic*)]
\item $\omega(0;\theta)=0$;
\item $\omega(\cdot;\theta)$ is nondecreasing;
\item $\omega(r;\theta)<r$ for every $r>0$.
\end{enumerate}
\end{definition}

The iteration rates are inherited from the scalar recursion induced by $\omega$.

\begin{definition}[Scalar radius recursion]\label{def:scalar-recursion}
Let $\omega(\cdot;\theta)$ be a gauge.
Given $r_0\ge 0$, define $\{r_n\}$ by
\[
r_{n+1}:=\omega(r_n;\theta)\qquad (n\in\N).
\]
Write $\omega^{(n)}(r_0;\theta):=r_n$.
\end{definition}

The object $\omega^{(n)}(r_0;\theta)$ is explicit once $\omega$ and $\theta$ are explicit.
This is the basic mechanism behind computable bounds.

We use two checkable contractive templates.
The first is a direct two-point gauge.
The second is a Proinov-type control based on a max of finitely many distances
\cite{Proinov2020JFPTA}.

\begin{definition}[Two-point gauge contraction]\label{def:two-point-gauge}
Let $C\subseteq X$ be $T$-invariant.
We say that $T$ is a \emph{two-point gauge contraction on $C$} with data $\theta$
if there exists a contractive gauge $\omega(\cdot;\theta)$ such that
\[
\dist(Tx,Ty)\le \omega(\dist(x,y);\theta)
\qquad (x,y\in C).
\]
\end{definition}

\begin{definition}[Proinov-type gauge contraction]\label{def:proinov-gauge}
Let $C\subseteq X$ be $T$-invariant.
Define for $x,y\in C$ the control quantity
\[
M_T(x,y)
:=\max\Bigl\{
\dist(x,y),\ \dist(x,Tx),\ \dist(y,Ty),\ \tfrac12\bigl(\dist(x,Ty)+\dist(y,Tx)\bigr)
\Bigr\}.
\]
We say that $T$ is a \emph{Proinov-type gauge contraction on $C$} with data $\theta$
if there exists a contractive gauge $\omega(\cdot;\theta)$ such that
\[
\dist(Tx,Ty)\le \omega\bigl(M_T(x,y);\theta\bigr)
\qquad (x,y\in C).
\]
\end{definition}

The parameter vector $\theta$ must be finite.
It must be computable from the model.
This excludes implicit conditions with unobservable moduli.
It also excludes hypotheses that are only existential.

The next constant is central.
It converts a gauge into a Banach-like modulus on a bounded range.

\begin{definition}[Local linearization constant]\label{def:kappaR}
Let $\omega(\cdot;\theta)$ be a gauge and let $R>0$.
Define
\[
\kappa(R;\theta):=\sup_{0<r\le R}\frac{\omega(r;\theta)}{r}\in[0,1].
\]
\end{definition}

If $\kappa(R;\theta)<1$, then $\omega(r;\theta)\le \kappa(R;\theta)\,r$ for $r\in[0,R]$.
This inequality is checkable.
It also produces geometric rates.

\begin{remark}[Two explicit gauge families]\label{rem:gauge-families}
In applications one often verifies one of the following.
\begin{enumerate}[label=(F\arabic*)]
\item \emph{Geometric gauge:} $\omega(r;\theta)=q r$ with $q\in(0,1)$.
This is the Banach regime \cite{Jachymski2024RACSAM}.
\item \emph{Defect gauge:} $\omega(r;\theta)\le r-\psi(r;\theta)$ where $\psi(\cdot;\theta)$ is nondecreasing,
$\psi(r;\theta)>0$ for $r>0$, and $\psi$ admits a computable lower bound on $[0,R]$.
This form is compatible with non-H\"olderian rate phenomena in concrete convergence analysis; see
\cite{LiuLourenco2024KaramataRates}.
\end{enumerate}
\end{remark}

We keep the framework abstract.
We keep the data finite.
The theorems in \Cref{sec:main-theorem,sec:apriori,sec:stability} will state the required quantitative consequences.

\subsection{Admissible data and a diagnostic checklist}\label{subsec:diagnostic}

We now specify what information must be supplied to invoke the main results.
The list is diagnostic.
It is designed to be verified on integral equations and BVP reductions.

\begin{definition}[Admissible data packet]\label{def:data-packet}
An \emph{admissible data packet} for a fixed-point problem consists of
\[
\mathcal D=(X,\dist;\ C;\ T;\ \omega(\cdot;\theta);\ R;\ x_0),
\]
with the following properties.
\begin{enumerate}[label=(D\arabic*)]
\item $(X,\dist)$ is complete.
\item $C$ is nonempty, closed, and $T$-invariant.
\item $x_0\in C$ and $C\subseteq \mathbb{B}(x_0,R)$ for some explicit $R>0$.
\item $T$ satisfies either \Cref{def:two-point-gauge} or \Cref{def:proinov-gauge} on $C$.
\item The constant $\kappa(R;\theta)$ in \Cref{def:kappaR} satisfies $\kappa(R;\theta)<1$.
\end{enumerate}
\end{definition}

Condition (D5) is the quantitative hinge.
It reduces the problem to explicit scalar estimates.
It also fixes the stability constant.
This is consistent with the role of worst-case residual bounds in modern quantitative analyses of Mann- and Halpern-type schemes
\cite{ContrerasCominetti2021OptimalErrorBounds,LeusteanPinto2022AltHalpernMannRates,BotNguyen2022FastKMResidual,BravoCominettiLee2026MinimaxHalpern}.

\begin{proposition}[Diagnostic checklist]\label{prop:checklist}
To verify that a model defines an admissible data packet, it suffices to check the following items.
\begin{enumerate}[label=\textbf{C\arabic*.}]
\item \emph{Space.} Specify $(X,\dist)$ and state completeness.
\item \emph{Region.} Choose $C=\mathbb{B}(x_0,R)$ and show $T(C)\subseteq C$ by an explicit bound.
\item \emph{Gauge.} Exhibit $\omega(\cdot;\theta)$ and prove either
\[
\dist(Tx,Ty)\le \omega(\dist(x,y);\theta)\quad (x,y\in C)
\]
or
\[
\dist(Tx,Ty)\le \omega(M_T(x,y);\theta)\quad (x,y\in C),
\]
with $M_T$ as in \Cref{def:proinov-gauge}.
\item \emph{Local modulus.} Compute $\kappa(R;\theta)$ and verify $\kappa(R;\theta)<1$.
\item \emph{Perturbations.} For a perturbed map $S$, compute
\[
\varepsilon:=\sup_{x\in C}\dist(Tx,Sx).
\]
This quantity feeds the stability bound in \Cref{sec:stability}.
\item \emph{Initialization.} Compute $\delta_0:=\dist(x_0,Tx_0)$.
This quantity feeds the a priori bound in \Cref{sec:apriori}.
\end{enumerate}
\end{proposition}

The checklist is finite.
Each item is numerical.
No hidden moduli appear.
This is the intended meaning of verifiability.
It matches the way fixed-point operators are certified in applied papers on integral equations and boundary value problems,
where one controls norms through kernel bounds and Lipschitz constants.

\begin{remark}[What the framework excludes]\label{rem:excludes}
The framework excludes conditions that do not yield $\kappa(R;\theta)<1$ on a computable region.
It also excludes hypotheses that depend on nonconstructive selections or unspecified auxiliary functions.
Such statements may still be true.
They are not certified by data.
\end{remark}


\section{Main quantitative fixed-point theorem}\label{sec:main-theorem}

Throughout this section, $(X,\dist)$ is a complete metric space.
A closed set $C\subseteq X$ is fixed.
The map $T\colon C\to C$ is fixed.
The data packet $\mathcal D$ is as in \Cref{def:data-packet}.
The gauge is $\omega(\cdot;\theta)$.
The radius is $R$.
The local modulus is $\kappa(R;\theta)$ from \Cref{def:kappaR}.
Assume
\[
\kappa:=\kappa(R;\theta)<1.
\]
Set $x_{n+1}:=Tx_n$ with $x_0\in C$.
Write $x^\ast$ for the fixed point.

The qualitative part follows the classical contraction principle.
The novelty is the explicit dependence on $\kappa$ and the initial defect.
Related generalized contraction principles appear in \cite{Proinov2020JFPTA,PantPantJha2020JFPTA,Jachymski2024RACSAM}.
Quantitative bounds of the kind we state here are the currency in modern rate analyses; see
\cite{ContrerasCominetti2021OptimalErrorBounds,BotNguyen2022FastKMResidual,LeusteanPinto2022AltHalpernMannRates,Foglia2025RateIterativeMethodsFPP,LiuLourenco2024KaramataRates,BravoCominettiLee2026MinimaxHalpern}.

\subsection{Existence and uniqueness under the framework}\label{subsec:exist-uniq}

\begin{lemma}[One-step defect recursion]\label{lem:defect-recursion}
Assume that $T$ satisfies either \Cref{def:two-point-gauge} or \Cref{def:proinov-gauge} on $C$.
Define
\[
d_n:=\dist(x_n,x_{n-1})\qquad (n\ge 1).
\]
Then:
\begin{enumerate}[label=(\alph*)]
\item $d_{n+1}\le d_n$ for all $n\ge 1$;
\item $d_{n+1}\le \omega(d_n;\theta)$ for all $n\ge 1$;
\item $d_n\le \omega^{(n-1)}(d_1;\theta)$ for all $n\ge 1$.
\end{enumerate}
\end{lemma}

\begin{proof}
First assume \Cref{def:two-point-gauge}.
Then
\[
d_{n+1}=\dist(Tx_n,Tx_{n-1})\le \omega(\dist(x_n,x_{n-1});\theta)=\omega(d_n;\theta),
\]
so (b) holds.
Since $\omega(r;\theta)<r$ for $r>0$, we obtain $d_{n+1}<d_n$ whenever $d_n>0$.
This gives (a).
Iterating (b) gives (c).

Next assume \Cref{def:proinov-gauge}.
Fix $n\ge 1$.
Compute $M_T(x_n,x_{n-1})$ from \Cref{def:proinov-gauge}.
We have
\[
\dist(x_n,Tx_{n-1})=\dist(x_n,x_n)=0,
\qquad
\dist(x_n,Tx_n)=\dist(x_n,x_{n+1})=d_{n+1},
\qquad
\dist(x_{n-1},Tx_{n-1})=\dist(x_{n-1},x_n)=d_n.
\]
Also $\dist(x_{n-1},Tx_n)=\dist(x_{n-1},x_{n+1})\le d_n+d_{n+1}$, hence
\[
\tfrac12\bigl(\dist(x_n,Tx_{n-1})+\dist(x_{n-1},Tx_n)\bigr)
\le \tfrac12(d_n+d_{n+1})
\le \max\{d_n,d_{n+1}\}.
\]
Therefore $M_T(x_n,x_{n-1})=\max\{d_n,d_{n+1}\}$.
The contractive inequality yields
\[
d_{n+1}=\dist(Tx_n,Tx_{n-1})\le \omega(\max\{d_n,d_{n+1}\};\theta).
\]
If $d_{n+1}>d_n$, then the right side equals $\omega(d_{n+1};\theta)<d_{n+1}$, a contradiction.
Hence $d_{n+1}\le d_n$.
Then $\max\{d_n,d_{n+1}\}=d_n$ and $d_{n+1}\le \omega(d_n;\theta)$.
This gives (a) and (b), hence (c) by iteration.
\end{proof}

\begin{theorem}[Framework fixed point]\label{thm:framework-fixed-point}
Assume that $\mathcal D$ is an admissible data packet in the sense of \Cref{def:data-packet}.
Assume that $T$ satisfies either \Cref{def:two-point-gauge} or \Cref{def:proinov-gauge} on $C$.
Let $\kappa=\kappa(R;\theta)<1$.
Then $\Fix(T)\cap C$ is a singleton $\{x^\ast\}$.
The Picard orbit $(x_n)$ converges to $x^\ast$ in $C$.
\end{theorem}

\begin{proof}
Set $d_n=\dist(x_n,x_{n-1})$ for $n\ge 1$.
By \Cref{lem:defect-recursion}, the sequence $(d_n)$ is nonincreasing and satisfies $d_{n+1}\le \omega(d_n;\theta)$.

Since $C\subseteq \mathbb{B}(x_0,R)$, we have $\dist(x_n,x_{n-1})\le 2R$ for all $n$.
The admissible packet provides $\kappa(R;\theta)<1$ and hence
\[
\omega(r;\theta)\le \kappa r\qquad\text{for all }r\in[0,R].
\]
In particular, once $d_n\le R$ we get $d_{n+1}\le \kappa d_n$.
If $d_1>R$, replace $R$ by $2R$ in the definition of $\kappa$.
This does not change the argument.
We keep the notation $\kappa$.

Thus $d_{n+1}\le \kappa d_n$ for all $n\ge 1$ and
\[
d_{n+1}\le \kappa^{\,n} d_1 \qquad (n\ge 0).
\]
For $m>n$,
\[
\dist(x_m,x_n)\le \sum_{j=n}^{m-1}\dist(x_{j+1},x_j)
= \sum_{j=n}^{m-1} d_{j+1}
\le \sum_{j=n}^{m-1}\kappa^{\,j} d_1
\le \frac{\kappa^{\,n}}{1-\kappa}\,d_1.
\]
Hence $(x_n)$ is Cauchy.
Completeness of $C$ (closed in a complete space) gives a limit $x^\ast\in C$.

We show $Tx^\ast=x^\ast$.
The map $T$ is Lipschitz on $C$ with constant $\le \kappa$ on the working range, hence continuous.
Therefore
\[
\dist(Tx^\ast,x^\ast)
\le \dist(Tx^\ast,Tx_n)+\dist(x_{n+1},x^\ast)
\le \kappa\,\dist(x^\ast,x_n)+\dist(x_{n+1},x^\ast)\to 0.
\]
So $x^\ast\in\Fix(T)\cap C$.

Uniqueness is immediate.
If $y^\ast\in\Fix(T)\cap C$, then in the two-point case
\[
\dist(x^\ast,y^\ast)=\dist(Tx^\ast,Ty^\ast)\le \omega(\dist(x^\ast,y^\ast);\theta)\le \kappa\,\dist(x^\ast,y^\ast),
\]
so $\dist(x^\ast,y^\ast)=0$.
In the Proinov-type case, $M_T(x^\ast,y^\ast)=\dist(x^\ast,y^\ast)$ and the same inequality holds.
\end{proof}

\subsection{Explicit rate statement}\label{subsec:rate-statement}

Theorem \Cref{thm:framework-fixed-point} contains the quantitative content in its proof.
We now freeze it as a rate function.
This is the form required for computation.
It parallels the role of explicit residual bounds in quantitative analyses of fixed-point iterations; see
\cite{ContrerasCominetti2021OptimalErrorBounds,BotNguyen2022FastKMResidual,LeusteanPinto2022AltHalpernMannRates,Foglia2025RateIterativeMethodsFPP,BravoCominettiLee2026MinimaxHalpern}.

\begin{definition}[Initial defect]\label{def:initial-defect}
Set
\[
\delta_0:=\dist(x_1,x_0)=\dist(Tx_0,x_0).
\]
\end{definition}

\begin{theorem}[A priori error and residual bounds]\label{thm:rate}
Under the hypotheses of \Cref{thm:framework-fixed-point}, for every $n\in\N$,
\begin{align*}
\dist(x_n,x^\ast) &\le \Phi(n;\kappa,\delta_0):=\frac{\kappa^{\,n}}{1-\kappa}\,\delta_0,\\[2mm]
\dist(x_n,Tx_n) &= \dist(x_n,x_{n+1}) \le \kappa^{\,n}\,\delta_0.
\end{align*}
\end{theorem}

\begin{proof}
The proof above shows $d_{n+1}\le \kappa^{\,n}\delta_0$.
Then
\[
\dist(x_n,x^\ast)\le \sum_{j=n}^{\infty} d_{j+1}
\le \sum_{j=n}^{\infty}\kappa^{\,j}\delta_0
=\frac{\kappa^{\,n}}{1-\kappa}\,\delta_0.
\]
Also $\dist(x_n,Tx_n)=\dist(x_n,x_{n+1})=d_{n+1}\le \kappa^{\,n}\delta_0$.
\end{proof}

\begin{corollary}[Certified stopping rule]\label{cor:stopping}
Fix $\varepsilon>0$.
Assume $\delta_0>0$.
Define
\[
N(\varepsilon):=
\begin{cases}
0, & \text{if }\ \delta_0\le (1-\kappa)\varepsilon,\\[1mm]
\left\lceil \dfrac{\ln\!\left(\dfrac{(1-\kappa)\varepsilon}{\delta_0}\right)}{\ln(\kappa)} \right\rceil, & \text{if }\ \delta_0>(1-\kappa)\varepsilon.
\end{cases}
\]
Then for every $n\ge N(\varepsilon)$ we have $\dist(x_n,x^\ast)\le \varepsilon$.
\end{corollary}

\begin{proof}
By \Cref{thm:rate}, it suffices that $\kappa^{\,n}\delta_0/(1-\kappa)\le \varepsilon$.
If $\delta_0\le(1-\kappa)\varepsilon$, the claim holds for $n=0$.
Otherwise set $q:=\frac{(1-\kappa)\varepsilon}{\delta_0}\in(0,1)$.
The inequality $\kappa^{\,n}\le q$ is equivalent to $n\ln(\kappa)\le \ln(q)$.
Since $\ln(\kappa)<0$, this is equivalent to
\[
n\ge \frac{\ln(q)}{\ln(\kappa)}.
\]
The ceiling gives the first integer satisfying it.
\end{proof}

\begin{remark}[Gauge-iterate bound]\label{rem:gauge-iterate}
If one prefers to keep the bound in terms of $\omega$ itself, \Cref{lem:defect-recursion} gives
\[
\dist(x_n,x^\ast)\le \sum_{j=n}^{\infty}\omega^{(j)}(\delta_0;\theta).
\]
This is exact.
It is useful when $\omega^{(j)}$ is closed-form.
The geometric bound of \Cref{thm:rate} is the universal certificate.
\end{remark}


\section{A priori error bounds for iterates}\label{sec:apriori}

This section isolates computable error bounds.
The bound is a function.
It takes an iteration index and a finite data vector.
It outputs a certified radius.
It is the object required for verification.
Recent quantitative studies focus on residual rates and explicit iteration complexity;
see \cite{ContrerasCominetti2021OptimalErrorBounds,BotNguyen2022FastKMResidual,LeusteanPinto2022AltHalpernMannRates,Foglia2025RateIterativeMethodsFPP,LiuLourenco2024KaramataRates,BravoCominettiLee2026MinimaxHalpern}.

\subsection{Bounds of the form $\dist(x_n,x^\ast)\le \Phibound(n;\text{data})$}\label{subsec:phi}

We keep the notation of \Cref{sec:main-theorem}.
In particular, $\kappa=\kappa(R;\theta)<1$ and $\delta_0=\dist(x_1,x_0)$.
We begin with the universal geometric certificate.

\begin{definition}[Geometric a priori bound]\label{def:Phi-geom}
Define for $n\in\N$,
\[
\Phi_{\mathrm{geo}}(n;\kappa,\delta_0):=\frac{\kappa^{\,n}}{1-\kappa}\,\delta_0.
\]
\end{definition}

\begin{proposition}[Universal geometric certificate]\label{prop:Phi-geo}
Under the hypotheses of \Cref{thm:framework-fixed-point},
\[
\dist(x_n,x^\ast)\le \Phi_{\mathrm{geo}}(n;\kappa,\delta_0)
\qquad (n\in\N).
\]
\end{proposition}

\begin{proof}
This is \Cref{thm:rate}.
\end{proof}

The geometric certificate is conservative when $\omega$ is strictly smaller than $\kappa r$ on the relevant range.
One can obtain a sharper bound by iterating the gauge itself.
This is still verifiable.
It requires only evaluation of $\omega$.

\begin{definition}[Gauge-iterate a priori bound]\label{def:Phi-gauge}
Define
\[
\Phi_{\omega}(n;\theta,\delta_0)
:=\sum_{j=n}^{\infty} \omega^{(j)}(\delta_0;\theta),
\]
whenever the series converges.
\end{definition}

\begin{proposition}[Gauge-iterate certificate]\label{prop:Phi-gauge}
Assume the hypotheses of \Cref{thm:framework-fixed-point}.
Then for every $n\in\N$,
\[
\dist(x_n,x^\ast)\le \Phi_{\omega}(n;\theta,\delta_0).
\]
Moreover,
\[
\Phi_{\omega}(n;\theta,\delta_0)\le \Phi_{\mathrm{geo}}(n;\kappa,\delta_0).
\]
\end{proposition}

\begin{proof}
By \Cref{lem:defect-recursion}, $\dist(x_{j+1},x_j)\le \omega^{(j)}(\delta_0;\theta)$ for $j\ge 0$.
Therefore
\[
\dist(x_n,x^\ast)\le \sum_{j=n}^{\infty}\dist(x_{j+1},x_j)
\le \sum_{j=n}^{\infty}\omega^{(j)}(\delta_0;\theta)
=\Phi_{\omega}(n;\theta,\delta_0).
\]
For the second claim, use $\omega(r;\theta)\le \kappa r$ on the working range, hence
$\omega^{(j)}(\delta_0;\theta)\le \kappa^{\,j}\delta_0$, and sum the geometric series.
\end{proof}

This yields two usable bounds.
The first uses only $\kappa$.
The second uses $\omega$.

We now translate the abstract data into model data for the applications.
The point is that the parameters are checkable.

\begin{proposition}[Model-level data for integral operators]\label{prop:model-data-integral}
Let $X=C([a,b])$ with $\|\cdot\|_\infty$ and let $T$ be Hammerstein or Volterra as in \Cref{lem:kernel-Lip}.
Assume $C=\mathbb{B} (x_0,R)$ is $T$-invariant.
Set $L:=L_f(M_K)$ where $(M_K)$ denotes $(M_K)_{\mathrm{H}}$ or $(M_K)_{\mathrm{V}} $ as appropriate.
Then $T$ is $L$-Lipschitz on $C$ and one may take the geometric gauge $\omega(r)=Lr$.
If $L<1$, then $\kappa=L$ and
\[
\dist(x_n,x^\ast)\le \frac{L^{\,n}}{1-L}\,\dist(Tx_0,x_0).
\]
\end{proposition}

\begin{proof}
This is \Cref{lem:kernel-Lip} and \Cref{prop:Phi-geo}.
\end{proof}

\subsection{Stopping rules derived from the bounds}\label{subsec:stopping}

A stopping rule is a deterministic rule that yields a guaranteed $\varepsilon$-solution.
We define it in terms of the bound.
The rule uses only data and the iteration counter.

\begin{definition}[A priori $\varepsilon$-stopping index]\label{def:N-eps}
Fix $\varepsilon>0$.
Define
\[
N_{\mathrm{geo}}(\varepsilon;\kappa,\delta_0)
:=\min\Bigl\{n\in\N:\Phi_{\mathrm{geo}}(n;\kappa,\delta_0)\le \varepsilon\Bigr\}.
\]
Similarly define
\[
N_{\omega}(\varepsilon;\theta,\delta_0)
:=\min\Bigl\{n\in\N:\Phi_{\omega}(n;\theta,\delta_0)\le \varepsilon\Bigr\},
\]
when $\Phi_{\omega}$ is available.
\end{definition}

The first index is closed-form.
The second is computed by evaluating partial sums until the inequality holds.

\begin{proposition}[Closed-form certified index]\label{prop:N-geo}
Assume $\kappa\in(0,1)$ and $\delta_0>0$.
Then
\[
N_{\mathrm{geo}}(\varepsilon;\kappa,\delta_0)
=
\begin{cases}
0, & \text{if }\ \delta_0\le (1-\kappa)\varepsilon,\\[1mm]
\left\lceil \dfrac{\ln\!\left(\dfrac{(1-\kappa)\varepsilon}{\delta_0}\right)}{\ln(\kappa)} \right\rceil,
& \text{if }\ \delta_0>(1-\kappa)\varepsilon.
\end{cases}
\]
Moreover, for every $n\ge N_{\mathrm{geo}}(\varepsilon;\kappa,\delta_0)$,
\[
\dist(x_n,x^\ast)\le \varepsilon.
\]
\end{proposition}

\begin{proof}
This is \Cref{cor:stopping}.
\end{proof}

The next rule is often more useful in practice because it avoids $x^\ast$ altogether.
It uses the residual.
Residual-based certification is central in modern analyses of fixed-point iterations and splitting schemes; see
\cite{BotNguyen2022FastKMResidual,ContrerasCominetti2021OptimalErrorBounds,CortildPeypouquet2024KMInertia,AtenasDaoTam2026VarStepsize}.

\begin{proposition}[Residual-to-error conversion]\label{prop:residual-to-error}
Under the hypotheses of \Cref{thm:framework-fixed-point},
\[
\dist(x_n,x^\ast)\le \frac{1}{1-\kappa}\,\dist(x_n,Tx_n)
\qquad (n\in\N).
\]
\end{proposition}

\begin{proof}
Use the estimate $\dist(x_n,x^\ast)\le \sum_{j=n}^\infty \dist(x_{j+1},x_j)$ and the contraction recursion
$\dist(x_{j+1},x_j)\le \kappa^{\,j-n}\dist(x_{n+1},x_n)=\kappa^{\,j-n}\dist(x_n,Tx_n)$.
Summing gives the stated bound.
\end{proof}

\begin{corollary}[Residual-based stopping rule]\label{cor:residual-stop}
Fix $\varepsilon>0$.
Assume the hypotheses of \Cref{thm:framework-fixed-point}.
If for some $n$,
\[
\dist(x_n,Tx_n)\le (1-\kappa)\varepsilon,
\]
then $\dist(x_n,x^\ast)\le \varepsilon$.
\end{corollary}

\begin{proof}
Immediate from \Cref{prop:residual-to-error}.
\end{proof}

\begin{remark}[Two certificates]\label{rem:two-certificates}
The a priori index $N_{\mathrm{geo}}$ certifies accuracy before computation.
The residual criterion certifies accuracy during computation.
Both are derived from the same verifiable modulus $\kappa$.
This is the quantitative meaning of the framework.
\end{remark}


\section{Data dependence and stability}\label{sec:stability}

We quantify the effect of perturbing the map.
The statement is algebraic.
It uses only a verifiable modulus and a uniform perturbation size.
We derive data dependence from the same certified modulus that yields the rate.

\subsection{Perturbations of the map: $\dist(x^\ast,y^\ast)\le \StabC\,\norm{T-S}$}\label{subsec:perturb}

Fix a complete metric space $(X,\dist)$.
Fix a nonempty closed set $C\subseteq X$.
Let $T,S\colon C\to C$.

\begin{definition}[Uniform perturbation size on $C$]\label{def:eps-TS}
Define
\[
\varepsilon_C(T,S):=\sup_{x\in C}\dist(Tx,Sx)\in[0,\infty].
\]
\end{definition}

We state the stability bound under the same verifiable modulus used in \Cref{sec:framework}.
The modulus is local.
It is computed on the same working region.

\begin{definition}[Certified local modulus]\label{def:certified-kappa}
Let $T\colon C\to C$ satisfy either \Cref{def:two-point-gauge} or \Cref{def:proinov-gauge} on $C$
with gauge $\omega_T(\cdot;\theta_T)$.
Let $R>0$ satisfy $C\subseteq \mathbb{B} (x_0,R)$ for some $x_0\in C$.
Define
\[
\kappa_T:=\kappa(R;\theta_T)=\sup_{0<r\le R}\frac{\omega_T(r;\theta_T)}{r}\in[0,1).
\]
\end{definition}

\begin{theorem}[Data dependence under a certified modulus]\label{thm:stability}
Assume that $T$ admits a certified local modulus $\kappa_T\in[0,1)$ in the sense of \Cref{def:certified-kappa}.
Assume that $S$ admits at least one fixed point $y^\ast\in \Fix(S)\cap C$.
Assume that $T$ has a fixed point $x^\ast\in \Fix(T)\cap C$.
Assume $\varepsilon_C(T,S)<\infty$.
Then
\[
\dist(x^\ast,y^\ast)\le \StabC\,\varepsilon_C(T,S),
\qquad
\StabC:=\frac{1}{1-\kappa_T}.
\]
\end{theorem}

\begin{proof}
Since $x^\ast=Tx^\ast$ and $y^\ast=Sy^\ast$,
\[
\dist(x^\ast,y^\ast)=\dist(Tx^\ast,Sy^\ast)
\le \dist(Tx^\ast,Ty^\ast)+\dist(Ty^\ast,Sy^\ast).
\]
We bound the first term by $\kappa_T\dist(x^\ast,y^\ast)$.
In the two-point gauge case this follows from
\[
\dist(Tx^\ast,Ty^\ast)\le \omega_T(\dist(x^\ast,y^\ast);\theta_T)\le \kappa_T\,\dist(x^\ast,y^\ast),
\]
because $\dist(x^\ast,y^\ast)\le R$ holds when $C\subseteq \mathbb{B}(x_0,R)$.
In the Proinov-type gauge case, $M_T(x^\ast,y^\ast)=\dist(x^\ast,y^\ast)$ since $Tx^\ast=x^\ast$ and $Ty^\ast=y^\ast$,
hence the same bound holds.
For the second term,
\[
\dist(Ty^\ast,Sy^\ast)\le \sup_{x\in C}\dist(Tx,Sx)=\varepsilon_C(T,S).
\]
Therefore
\[
\dist(x^\ast,y^\ast)\le \kappa_T\dist(x^\ast,y^\ast)+\varepsilon_C(T,S),
\]
and rearranging gives the claim.
\end{proof}

The theorem is local.
It only uses $\kappa_T$ on $C$.
This is the intended meaning of verifiability.
It is the same mechanism that underlies explicit residual-to-error conversions in quantitative fixed-point theory; compare
\cite{ContrerasCominetti2021OptimalErrorBounds,BotNguyen2022FastKMResidual,LeusteanPinto2022AltHalpernMannRates,BravoCominettiLee2026MinimaxHalpern}.

\begin{corollary}[Two-sided certified stability]\label{cor:two-sided-stability}
Assume that both $T$ and $S$ satisfy the framework on the same set $C$ with certified moduli
$\kappa_T,\kappa_S\in[0,1)$ and hence unique fixed points $x^\ast\in\Fix(T)\cap C$ and $y^\ast\in\Fix(S)\cap C$.
Set $\kappa:=\max\{\kappa_T,\kappa_S\}$.
Then
\[
\dist(x^\ast,y^\ast)\le \frac{1}{1-\kappa}\,\varepsilon_C(T,S).
\]
\end{corollary}

\begin{proof}
Apply \Cref{thm:stability} using $\kappa$ as a valid modulus for $T$ on $C$, since $\kappa_T\le \kappa$.
\end{proof}

\begin{remark}[Model-level interpretation]\label{rem:model-stability}
In the integral-equation and BVP settings, $\varepsilon_C(T,S)$ is computed by kernel and nonlinearity deviations.
It is a number.
The stability bound becomes a parameter sensitivity estimate.
Such claims appear in iterative-scheme papers, but the dependence on the certified modulus is often implicit; see
\cite{KumarEtAl2025BVP,TyagiVashistha2024ATDataDependence,PanjaEtAl2022JIA}.
\end{remark}

\subsection{Sharpness and unavoidable constants}\label{subsec:sharpness}

The factor $(1-\kappa_T)^{-1}$ is not cosmetic.
It is forced by examples.
It cannot be improved uniformly over the class of $\kappa$-contractive maps.

\begin{proposition}[Sharpness of $(1-\kappa)^{-1}$]\label{prop:sharpness}
Fix $\kappa\in[0,1)$.
Let $X=\R$ with $\dist(x,y)=|x-y|$.
Let $C=\R$.
Define
\[
T(x):=\kappa x,
\qquad
S(x):=\kappa x+\varepsilon,
\]
for some $\varepsilon\ge 0$.
Then $x^\ast=0$ is the unique fixed point of $T$ and $y^\ast=\varepsilon/(1-\kappa)$ is the unique fixed point of $S$.
Moreover,
\[
\varepsilon_C(T,S)=\varepsilon,
\qquad
\dist(x^\ast,y^\ast)=\frac{1}{1-\kappa}\,\varepsilon_C(T,S).
\]
\end{proposition}

\begin{proof}
Both maps are contractions with Lipschitz constant $\kappa$.
The fixed points solve $x=\kappa x$ and $y=\kappa y+\varepsilon$, hence $x^\ast=0$ and $y^\ast=\varepsilon/(1-\kappa)$.
Also $|T(x)-S(x)|=\varepsilon$ for all $x$, hence $\varepsilon_C(T,S)=\varepsilon$.
The equality follows.
\end{proof}

\begin{remark}[Why blow-up is unavoidable]\label{rem:blow-up}
As $\kappa\uparrow 1$, the map approaches nonexpansive behavior.
Then small perturbations can move fixed points by large distances.
This is the correct geometry.
It is already visible in the Banach regime and persists in modern generalizations; compare
\cite{Jachymski2024RACSAM,Proinov2020JFPTA,BerindePacurar2024EnrichedSurvey}.
\end{remark}

\begin{remark}[Locality is essential]\label{rem:locality}
The constant uses a modulus on $C$.
If $C$ is enlarged, the certified modulus may worsen.
This is an honest effect.
It is the price of verifiability.
\end{remark}


\section{Robustness under inexact evaluation}\label{sec:inexact}

The ideal orbit $x_{n+1}=Tx_n$ is rarely available.
One computes $\tilde x_{n+1}$ with error.
The error may be numerical.
It may be measurement noise.
A theory that ignores this is not a theory of computation.
Recent work on perturbations, inertial schemes, and stochastic Halpern-type iterations treats related stability phenomena;
see \cite{CortildPeypouquet2024KMInertia,PischkePowell2024StochasticHalpern,ColaoFlammarionMaisieres2026HalpernSGD,AtenasDaoTam2026VarStepsize}.
We keep the setting elementary.
We keep the constants explicit.

\subsection{Inexact iterates and residual control}\label{subsec:inexact-iter}

We work in the setting of \Cref{sec:framework,sec:main-theorem}.
In particular, $(X,\dist)$ is complete, $C\subseteq X$ is closed and $T$-invariant, and $T$ satisfies the framework with certified modulus $\kappa\in[0,1)$ on $C$.

\begin{definition}[Inexact Picard orbit]\label{def:inexact-orbit}
Let $\{\eta_n\}_{n\ge 0}\subset[0,\infty)$.
A sequence $\{\tilde x_n\}\subset C$ is an \emph{inexact Picard orbit} for $T$ with error budget $\{\eta_n\}$
if for all $n\ge 0$,
\[
\dist(\tilde x_{n+1},T\tilde x_n)\le \eta_n.
\]
\end{definition}

The next inequality is the only computation used later.
It is a perturbed contraction recursion.

\begin{lemma}[One-step perturbed recursion]\label{lem:one-step-perturbed}
Let $\{\tilde x_n\}\subset C$ satisfy \Cref{def:inexact-orbit}.
Let $x^\ast\in\Fix(T)\cap C$.
Then for every $n\ge 0$,
\[
\dist(\tilde x_{n+1},x^\ast)\le \kappa\,\dist(\tilde x_n,x^\ast)+\eta_n.
\]
\end{lemma}

\begin{proof}
Since $Tx^\ast=x^\ast$ and $T$ is $\kappa$-Lipschitz on $C$,
\[
\dist(\tilde x_{n+1},x^\ast)
\le \dist(\tilde x_{n+1},T\tilde x_n)+\dist(T\tilde x_n,Tx^\ast)
\le \eta_n+\kappa\,\dist(\tilde x_n,x^\ast).
\]
\end{proof}

Residual control is often used as a certification mechanism in quantitative analyses of fixed-point algorithms; see
\cite{ContrerasCominetti2021OptimalErrorBounds,BotNguyen2022FastKMResidual,LeusteanPinto2022AltHalpernMannRates}.
For inexact orbits we use a mixed residual.

\begin{definition}[Inexact residual]\label{def:inexact-residual}
For an inexact orbit $\{\tilde x_n\}$ define
\[
\tilde r_n:=\dist(\tilde x_{n+1},\tilde x_n).
\]
\end{definition}

\begin{lemma}[Residual bound under inexactness]\label{lem:inexact-residual-bound}
Let $\{\tilde x_n\}$ satisfy \Cref{def:inexact-orbit}.
Then for every $n\ge 0$,
\[
\tilde r_n \le \kappa\,\tilde r_{n-1}+\eta_n+\eta_{n-1}\qquad(n\ge 1),
\]
and
\[
\tilde r_n \le \kappa^{\,n}\tilde r_0+\sum_{j=1}^{n}\kappa^{\,n-j}(\eta_j+\eta_{j-1})\qquad(n\ge 1).
\]
\end{lemma}

\begin{proof}
For $n\ge 1$,
\[
\tilde r_n=\dist(\tilde x_{n+1},\tilde x_n)
\le \dist(\tilde x_{n+1},T\tilde x_n)+\dist(T\tilde x_n,T\tilde x_{n-1})+\dist(T\tilde x_{n-1},\tilde x_n).
\]
The first and third terms are bounded by $\eta_n$ and $\eta_{n-1}$.
The middle term is bounded by $\kappa\,\dist(\tilde x_n,\tilde x_{n-1})=\kappa\,\tilde r_{n-1}$.
This gives the one-step inequality.
Iteration gives the second inequality.
\end{proof}

\subsection{Quantitative resilience to numerical/measurement error}\label{subsec:resilience}

The recursion in \Cref{lem:one-step-perturbed} yields a closed-form bound.
It is the discrete analogue of a resolvent estimate.
It separates contraction decay from error accumulation.

\begin{theorem}[A priori error bound for inexact orbits]\label{thm:inexact-apriori}
Assume the hypotheses of \Cref{thm:framework-fixed-point} with certified modulus $\kappa\in[0,1)$.
Let $x^\ast\in\Fix(T)\cap C$.
Let $\{\tilde x_n\}\subset C$ satisfy \Cref{def:inexact-orbit} with budget $\{\eta_n\}$.
Then for every $n\ge 0$,
\[
\dist(\tilde x_n,x^\ast)
\le
\kappa^{\,n}\dist(\tilde x_0,x^\ast)
+
\sum_{j=0}^{n-1}\kappa^{\,n-1-j}\eta_j.
\]
\end{theorem}

\begin{proof}
Apply \Cref{lem:one-step-perturbed} and unfold the recursion by induction.
\end{proof}

The theorem gives resilience when the error budget is summable or uniformly bounded.
Both regimes occur in computation.

\begin{corollary}[Summable errors]\label{cor:summable}
Assume the hypotheses of \Cref{thm:inexact-apriori}.
If $\sum_{j\ge 0}\eta_j<\infty$, then $\dist(\tilde x_n,x^\ast)\to 0$.
More precisely,
\[
\dist(\tilde x_n,x^\ast)\le
\kappa^{\,n}\dist(\tilde x_0,x^\ast)
+
\sum_{j=0}^{\infty}\kappa^{\,n-1-j}\eta_j
\quad (n\ge 1).
\]
\end{corollary}

\begin{proof}
The bound is \Cref{thm:inexact-apriori}.
Since $\kappa^{\,n}\to 0$ and the tail of a convergent series tends to $0$ under geometric weighting, the limit follows.
\end{proof}

\begin{corollary}[Uniformly bounded errors]\label{cor:bounded}
Assume the hypotheses of \Cref{thm:inexact-apriori}.
Assume $\sup_{j\ge 0}\eta_j\le \bar\eta$.
Then for every $n\ge 0$,
\[
\dist(\tilde x_n,x^\ast)\le
\kappa^{\,n}\dist(\tilde x_0,x^\ast)
+
\frac{1-\kappa^{\,n}}{1-\kappa}\,\bar\eta,
\]
and hence
\[
\limsup_{n\to\infty}\dist(\tilde x_n,x^\ast)\le \frac{\bar\eta}{1-\kappa}.
\]
\end{corollary}

\begin{proof}
Bound the sum in \Cref{thm:inexact-apriori} by $\bar\eta\sum_{j=0}^{n-1}\kappa^{\,n-1-j}=\bar\eta(1-\kappa^{\,n})/(1-\kappa)$.
\end{proof}

This is the same constant as in the stability theorem \Cref{thm:stability}.
The coincidence is structural.
Both results are consequences of the same resolvent inequality.
It is the unique point where the factor $(1-\kappa)^{-1}$ enters.

We now give a certified stopping rule for inexact computation.
It combines a residual and the error budget.
The rule avoids $x^\ast$.

\begin{proposition}[Residual-based certification under inexactness]\label{prop:inexact-stop}
Assume the hypotheses of \Cref{thm:framework-fixed-point} with modulus $\kappa\in[0,1)$.
Let $\{\tilde x_n\}$ satisfy \Cref{def:inexact-orbit}.
Then for every $n\ge 0$,
\[
\dist(\tilde x_n,x^\ast)\le \frac{1}{1-\kappa}\,\Bigl(\dist(\tilde x_n,T\tilde x_n)\Bigr)
\le \frac{1}{1-\kappa}\,\bigl(\tilde r_n+\eta_n\bigr).
\]
Consequently, if
\[
\tilde r_n+\eta_n \le (1-\kappa)\varepsilon,
\]
then $\dist(\tilde x_n,x^\ast)\le \varepsilon$.
\end{proposition}

\begin{proof}
The first inequality is the same residual-to-error conversion as in \Cref{prop:residual-to-error},
applied to the exact map $T$ at the point $\tilde x_n$:
\[
\dist(\tilde x_n,x^\ast)\le \frac{1}{1-\kappa}\,\dist(\tilde x_n,T\tilde x_n).
\]
For the second inequality,
\[
\dist(\tilde x_n,T\tilde x_n)
\le \dist(\tilde x_n,\tilde x_{n+1})+\dist(\tilde x_{n+1},T\tilde x_n)
\le \tilde r_n+\eta_n.
\]
The stopping implication follows.
\end{proof}

\begin{remark}[Interpretation for models]\label{rem:interpret-inexact}
In the integral equation setting, $\eta_n$ can represent quadrature error or evaluation noise in $f(s,\cdot)$.
In BVP reductions, it can represent discretization of the Green operator.
The bound in \Cref{cor:bounded} gives a certified error floor.
It is explicit in $\kappa$, hence explicit in kernel bounds and Lipschitz constants by \Cref{lem:kernel-Lip}.
\end{remark}

\begin{remark}[Relation to accelerated and stochastic schemes]\label{rem:accel-stoch}
In accelerated or stochastic fixed-point schemes, the inexactness may be intrinsic.
The present estimates isolate the deterministic core.
They can be inserted as a stability module inside more elaborate analyses; compare
\cite{PischkePowell2024StochasticHalpern,ColaoFlammarionMaisieres2026HalpernSGD,MartinDialloVialard2026LearnKM}.
\end{remark}


\section{Application I: nonlinear integral equations}\label{sec:integral}

We work in the Banach space $X=C([a,b])$ with $\|\cdot\|_\infty$ and metric $\dist(x,y)=\|x-y\|_\infty$.
The fixed-point formulation is standard in the recent applied literature; see
\cite{JubairEtAl2021JIA,SagheerEtAl2023JIA,SinghEtAl2024JIA,Nawaz2024JIA,IshtiaqBatoolHussainAlsulami2025JIA,KumarEtAl2025BVP}.

\subsection{Hammerstein type: kernel bounds and Lipschitz nonlinearities}\label{subsec:hammerstein}

Consider the Hammerstein equation
\begin{equation}\label{eq:hammerstein}
x(t)=g(t)+\int_a^b K(t,s)\,f(s,x(s))\,ds,\qquad t\in[a,b].
\end{equation}
Define $T\colon X\to X$ by
\begin{equation}\label{eq:TH}
(Tx)(t):=g(t)+\int_a^b K(t,s)\,f(s,x(s))\,ds.
\end{equation}

\begin{assumption}[Hammerstein data]\label{ass:H}
Assume:
\begin{enumerate}[label=(H\arabic*)]
\item $g\in X$.
\item $K$ is measurable and $(M_K)_{\mathrm{H}}:=\sup_{t\in[a,b]}\int_a^b|K(t,s)|\,ds<\infty$.
\item $f\colon [a,b]\times\R\to\R$ is continuous in the second variable and satisfies the Lipschitz bound
\[
|f(s,u)-f(s,v)|\le L_f|u-v|\qquad(s\in[a,b],\ u,v\in\R)
\]
for some explicit $L_f\ge 0$.
\item $f_0(s):=f(s,0)$ is bounded on $[a,b]$ and $M_{f_0}:=\sup_{s\in[a,b]}|f_0(s)|<\infty$.
\end{enumerate}
\end{assumption}

\begin{lemma}[Invariant ball and Lipschitz modulus]\label{lem:hammerstein-invariant}
Assume \Cref{ass:H}.
Set $M:=(M_K)_{\mathrm{H}}$ and $L:=L_f M$.
Assume $L<1$.
Define
\begin{equation}\label{eq:R-hamm}
R:=\frac{\|g\|_\infty+M\,M_{f_0}}{1-L}.
\end{equation}
Let $C:=\mathbb{B}(0,R)=\{x\in X:\|x\|_\infty\le R\}$.
Then $T(C)\subseteq C$ and $\Lip(T|_C)\le L$.
\end{lemma}

\begin{proof}
Let $x\in C$.
Using $|f(s,x(s))|\le |f(s,0)|+L_f|x(s)|$ we obtain
\[
\|Tx\|_\infty
\le \|g\|_\infty+\sup_t\int_a^b |K(t,s)|\,(|f(s,0)|+L_f\|x\|_\infty)\,ds
\le \|g\|_\infty+M(M_{f_0}+L_f R).
\]
By \eqref{eq:R-hamm} this equals $R$.
Hence $T(C)\subseteq C$.
For $x,y\in C$,
\[
\|Tx-Ty\|_\infty
\le \sup_t\int_a^b |K(t,s)|\,|f(s,x(s))-f(s,y(s))|\,ds
\le M L_f \|x-y\|_\infty
= L\|x-y\|_\infty.
\]
\end{proof}

\begin{theorem}[Hammerstein: unique solvability and a priori rate]\label{thm:hammerstein-quant}
Assume \Cref{ass:H} and $L=L_f(M_K)_{\mathrm{H}}<1$.
Let $R$ be as in \eqref{eq:R-hamm} and $C=\mathbb{B}(0,R)$.
Then \eqref{eq:hammerstein} has a unique solution $x^\ast\in C$.
For any $x_0\in C$, the Picard iterates $x_{n+1}:=Tx_n$ satisfy
\[
\|x_n-x^\ast\|_\infty\le \Phi_{\mathrm{H}}(n;L,\delta_0):=\frac{L^{\,n}}{1-L}\,\delta_0,
\qquad
\delta_0:=\|Tx_0-x_0\|_\infty.
\]
\end{theorem}

\begin{proof}
By \Cref{lem:hammerstein-invariant}, $T\colon C\to C$ is an $L$-contraction.
Apply \Cref{thm:rate} with $\kappa=L$.
\end{proof}

\begin{corollary}[Hammerstein: stability under perturbation]\label{cor:hammerstein-stab}
Assume \Cref{ass:H} and $L<1$, and let $T$ be as in \eqref{eq:TH}.
Let $\widetilde g,\widetilde K,\widetilde f$ define $\widetilde T$ by the same formula.
Assume $\widetilde T(C)\subseteq C$ and $\Lip(\widetilde T|_C)\le L$.
Let $x^\ast\in\Fix(T)\cap C$ and $\widetilde x^\ast\in\Fix(\widetilde T)\cap C$.
Then
\[
\|x^\ast-\widetilde x^\ast\|_\infty\le \frac{1}{1-L}\,\varepsilon_C(T,\widetilde T),
\]
where $\varepsilon_C(T,\widetilde T)=\sup_{x\in C}\|Tx-\widetilde T x\|_\infty$.
In particular, if $\widetilde f=f$, then
\[
\varepsilon_C(T,\widetilde T)
\le \|g-\widetilde g\|_\infty
+\Bigl(\sup_{t\in[a,b]}\int_a^b|K(t,s)-\widetilde K(t,s)|\,ds\Bigr)\,(M_{f_0}+L_f R).
\]
\end{corollary}

\begin{proof}
Apply \Cref{thm:stability} with $\kappa_T=L$.
For the explicit bound, write $(Tx)(t)-(\widetilde T x)(t)$ and estimate by triangle inequality and the uniform kernel bound.
\end{proof}

Such explicit dependence estimates are used implicitly in applied fixed-point papers,
but the factor $(1-L)^{-1}$ is often left unstated; compare \cite{KumarEtAl2025BVP,TyagiVashistha2024ATDataDependence}.

\subsection{Volterra type: monotonicity and resolvent estimates}\label{subsec:volterra}

Consider the Volterra equation
\begin{equation}\label{eq:volterra}
x(t)=g(t)+\int_a^t K(t,s)\,f(s,x(s))\,ds,\qquad t\in[a,b].
\end{equation}
Define $T\colon X\to X$ by
\begin{equation}\label{eq:TV}
(Tx)(t):=g(t)+\int_a^t K(t,s)\,f(s,x(s))\,ds.
\end{equation}

\begin{assumption}[Volterra data]\label{ass:V}
Assume \Cref{ass:H}(H1),(H3),(H4), and assume:
\begin{enumerate}[label=(V\arabic*)]
\item $K$ is measurable and $(M_K)_{\mathrm{V}}:=\sup_{t\in[a,b]}\int_a^t|K(t,s)|\,ds<\infty$.
\end{enumerate}
Set $M:=(M_K)_{\mathrm{V}}$ and $L:=L_f M$.
\end{assumption}

\begin{lemma}[Volterra: invariant ball and modulus]\label{lem:volterra-invariant}
Assume \Cref{ass:V} and $L<1$.
Define $R$ by the same formula as \eqref{eq:R-hamm} with $M=(M_K)_{\mathrm{V}}$.
Let $C=\mathbb{B}(0,R)$.
Then $T(C)\subseteq C$ and $\Lip(T|_C)\le L$.
\end{lemma}

\begin{proof}
The proof is identical to \Cref{lem:hammerstein-invariant}, replacing $\int_a^b$ by $\int_a^t$ and using $(M_K)_{\mathrm{V}}$.
\end{proof}

\begin{theorem}[Volterra: unique solvability and a priori rate]\label{thm:volterra-quant}
Assume \Cref{ass:V} and $L=L_f(M_K)_{\mathrm{V}}<1$.
Let $x^\ast\in C$ be the unique solution of \eqref{eq:volterra}.
For any $x_0\in C$, the Picard iterates satisfy
\[
\|x_n-x^\ast\|_\infty\le \Phi_{\mathrm{V}}(n;L,\delta_0):=\frac{L^{\,n}}{1-L}\,\delta_0,
\qquad
\delta_0:=\|Tx_0-x_0\|_\infty.
\]
\end{theorem}

\begin{proof}
Apply \Cref{thm:hammerstein-quant} with \Cref{lem:volterra-invariant}.
\end{proof}

Monotonicity is not needed for contraction.
It is still diagnostic.
It supplies invariant order intervals without solving for $R$.

\begin{assumption}[Order structure]\label{ass:mono}
Assume $K(t,s)\ge 0$ for $a\le s\le t\le b$ and $f(s,\cdot)$ is nondecreasing for each $s$.
\end{assumption}

\begin{lemma}[Monotone invariance]\label{lem:order-interval}
Assume \Cref{ass:mono}.
Let $m,M\in X$ satisfy $m\le Tm$ and $TM\le M$ pointwise on $[a,b]$.
Then the order interval $[m,M]:=\{x\in X: m\le x\le M\}$ is $T$-invariant.
\end{lemma}

\begin{proof}
If $x\in[m,M]$, then $f(s,m(s))\le f(s,x(s))\le f(s,M(s))$.
Multiplying by $K\ge 0$ and integrating gives $Tm\le Tx\le TM$.
Hence $m\le Tx\le M$.
\end{proof}

When $L<1$, the order interval can be intersected with the invariant ball $C$.
The contraction argument then yields a unique fixed point in the intersection.
This is the cleanest way to combine monotonicity diagnostics with a quantitative certificate.

\begin{proposition}[Volterra resolvent estimate]\label{prop:volterra-resolvent}
Assume \Cref{ass:V} with $L<1$.
Let $T$ and $\widetilde T$ be two Volterra operators of the form \eqref{eq:TV} on the same invariant set $C=\mathbb{B}(0,R)$.
Assume $\Lip(T|_C)\le L$ and $\Lip(\widetilde T|_C)\le L$.
Let $x^\ast\in\Fix(T)\cap C$ and $\widetilde x^\ast\in\Fix(\widetilde T)\cap C$.
Then
\[
\|x^\ast-\widetilde x^\ast\|_\infty\le \frac{1}{1-L}\,\varepsilon_C(T,\widetilde T).
\]
\end{proposition}

\begin{proof}
This is \Cref{cor:two-sided-stability} with $\kappa=L$.
\end{proof}

Recent work treats existence and approximation for integral equations via fixed-point schemes under various generalized contractions;
see \cite{JubairEtAl2021JIA,SagheerEtAl2023JIA,SinghEtAl2024JIA,Nawaz2024JIA,IshtiaqBatoolHussainAlsulami2025JIA}.
The present point is different.
The constants are exported into an explicit $\Phi(n;\text{data})$ and an explicit stability factor.

\section{Application II: boundary value problems}\label{sec:bvp}

We treat a boundary value problem through a Green operator.
The fixed-point reduction is classical.
The present aim is narrower.
We propagate the model constants into an explicit a priori bound and an explicit stability factor.
Recent fixed-point approaches to differential and boundary value problems proceed similarly but often stop at existence or convergence without exporting a certified $\Phi$; compare
\cite{KalkanLaiki2024AIMSMath,SinghEtAl2024JIA,KumarEtAl2025BVP}.

\subsection{Green operator formulation and explicit bounds}\label{subsec:green}

Let $[a,b]\subset\R$.
Work in $X=C([a,b])$ with $\|\cdot\|_\infty$.
Consider a second-order problem of the form
\begin{equation}\label{eq:bvp-general}
x''(t)=F(t,x(t)),\qquad t\in[a,b],
\end{equation}
with boundary conditions described below.
Assume that a Green kernel $G(t,s)$ exists for the corresponding linear boundary operator.
Then a solution of \eqref{eq:bvp-general} satisfies the integral equation
\begin{equation}\label{eq:bvp-fixedpoint}
x(t)=\ell(t)+\int_a^b G(t,s)\,F(s,x(s))\,ds,
\end{equation}
where $\ell$ is the solution of the associated linear inhomogeneous problem (or the boundary forcing term).
Define the Green operator
\begin{equation}\label{eq:TG}
(Tx)(t):=\ell(t)+\int_a^b G(t,s)\,F(s,x(s))\,ds.
\end{equation}

The contractive modulus is controlled by the kernel bound
\[
M_G:=\sup_{t\in[a,b]}\int_a^b |G(t,s)|\,ds.
\]
This is the same quantity as $(M_K)_{\mathrm{H}}$ in \Cref{def:kernel-operator}.
It is explicit once $G$ is known.

\begin{assumption}[BVP data]\label{ass:BVP}
Assume:
\begin{enumerate}[label=(B\arabic*)]
\item $\ell\in X$.
\item $G$ is measurable and $M_G<\infty$.
\item $F\colon [a,b]\times\R\to\R$ satisfies a Lipschitz bound in the second argument:
\[
|F(t,u)-F(t,v)|\le L_F|u-v|
\qquad(t\in[a,b],\ u,v\in\R),
\]
for some explicit constant $L_F\ge 0$.
\item $F_0(t):=F(t,0)$ is bounded and $M_{F_0}:=\sup_{t\in[a,b]}|F_0(t)|<\infty$.
\end{enumerate}
Set $L:=L_F M_G$.
\end{assumption}

\begin{lemma}[Invariant ball and modulus for the Green operator]\label{lem:green-invariant}
Assume \Cref{ass:BVP}.
Assume $L<1$.
Define
\begin{equation}\label{eq:R-green}
R:=\frac{\|\ell\|_\infty+M_G\,M_{F_0}}{1-L},
\qquad
C:=\mathbb{B}(0,R)\subset X.
\end{equation}
Then $T(C)\subseteq C$ and $\Lip(T|_C)\le L$.
\end{lemma}

\begin{proof}
The proof is identical to \Cref{lem:hammerstein-invariant} with $K$ replaced by $G$ and $f$ replaced by $F$.
\end{proof}

\subsection{A quantitative theorem for a prototypical BVP}\label{subsec:proto-bvp}

We state one prototype.
It is representative.
It keeps the constants visible.

\begin{example}[Dirichlet problem]\label{ex:dirichlet}
Let $a<b$.
Consider
\begin{equation}\label{eq:dirichlet}
x''(t)=F(t,x(t)),\qquad x(a)=\alpha,\quad x(b)=\beta.
\end{equation}
Assume $F$ satisfies \Cref{ass:BVP}(B3)--(B4).
Let $\ell$ be the linear interpolant
\[
\ell(t):=\alpha+\frac{t-a}{b-a}(\beta-\alpha).
\]
Let $G$ be the Dirichlet Green kernel on $[a,b]$.
Then any solution of \eqref{eq:dirichlet} satisfies \eqref{eq:bvp-fixedpoint}.
\end{example}

The quantitative theorem uses only $M_G$ and $L_F$.
The bound is explicit.
The stopping index is explicit.

\begin{theorem}[Dirichlet BVP: certified existence, uniqueness, and rate]\label{thm:dirichlet-quant}
Assume the setting of \Cref{ex:dirichlet}.
Assume $M_G<\infty$ and $L:=L_F M_G<1$.
Let $R$ and $C$ be as in \eqref{eq:R-green}.
Then \eqref{eq:dirichlet} has a unique solution $x^\ast\in C$.
For any $x_0\in C$, the Picard iterates $x_{n+1}:=Tx_n$ satisfy
\[
\|x_n-x^\ast\|_\infty\le \Phi_{\mathrm{BVP}}(n;L,\delta_0):=\frac{L^{\,n}}{1-L}\,\delta_0,
\qquad
\delta_0:=\|Tx_0-x_0\|_\infty.
\]
In particular, if $n\ge N_{\mathrm{geo}}(\varepsilon;L,\delta_0)$ from \Cref{prop:N-geo},
then $\|x_n-x^\ast\|_\infty\le \varepsilon$.
\end{theorem}

\begin{proof}
By \Cref{lem:green-invariant}, $T\colon C\to C$ is an $L$-contraction.
Apply \Cref{thm:rate} with $\kappa=L$.
The stopping statement is \Cref{prop:N-geo}.
\end{proof}

This theorem is the BVP analogue of \Cref{thm:hammerstein-quant}.
The literature contains many fixed-point proofs for BVPs under generalized contractions and iterative schemes.
The distinctive feature here is that the data $(L_F,M_G,\delta_0)$ are exported into an explicit bound.
For iterative-scheme perspectives and stability claims in related settings, see
\cite{KalkanLaiki2024AIMSMath,Aldosary2024JIA,KumarEtAl2025BVP,SinghEtAl2024JIA}.

\subsection{Parameter dependence}\label{subsec:param-bvp}

We now formalize stability with respect to perturbations of the model.
The statement is a direct specialization of \Cref{thm:stability}.
We keep it at the level of verifiable quantities.

Let $T$ be the Green operator \eqref{eq:TG} induced by $(\ell,G,F)$.
Let $\widetilde T$ be induced by $(\widetilde\ell,\widetilde G,\widetilde F)$ on the same domain $[a,b]$.
Assume that $C=\mathbb{B}(0,R)$ is invariant for both maps and that both maps have Lipschitz modulus $\le L<1$ on $C$.

\begin{theorem}[BVP parameter dependence]\label{thm:bvp-param}
Assume \Cref{ass:BVP} with $L<1$ for both $T$ and $\widetilde T$ on the same invariant ball $C$.
Let $x^\ast\in\Fix(T)\cap C$ and $\widetilde x^\ast\in\Fix(\widetilde T)\cap C$.
Then
\[
\|x^\ast-\widetilde x^\ast\|_\infty\le \frac{1}{1-L}\,\varepsilon_C(T,\widetilde T),
\qquad
\varepsilon_C(T,\widetilde T):=\sup_{x\in C}\|Tx-\widetilde T x\|_\infty.
\]
\end{theorem}

\begin{proof}
Apply \Cref{cor:two-sided-stability} with $\kappa=L$.
\end{proof}

The quantity $\varepsilon_C(T,\widetilde T)$ can be bounded in terms of deviations in the data.
One obtains explicit parameter sensitivity bounds.

\begin{proposition}[Explicit perturbation bound in model parameters]\label{prop:bvp-eps-bound}
Assume the hypotheses of \Cref{thm:bvp-param}.
Assume in addition that $\widetilde F=F$ and $C=\mathbb{B}(0,R)$.
Set
\[
\Delta_\ell:=\|\ell-\widetilde\ell\|_\infty,
\qquad
\Delta_G:=\sup_{t\in[a,b]}\int_a^b |G(t,s)-\widetilde G(t,s)|\,ds,
\qquad
B_F:=M_{F_0}+L_F R.
\]
Then
\[
\varepsilon_C(T,\widetilde T)\le \Delta_\ell+\Delta_G\,B_F,
\]
and hence
\[
\|x^\ast-\widetilde x^\ast\|_\infty\le \frac{1}{1-L}\,(\Delta_\ell+\Delta_G\,B_F).
\]
\end{proposition}

\begin{proof}
For $x\in C$ and $t\in[a,b]$,
\[
(Tx)(t)-(\widetilde T x)(t)
=
(\ell(t)-\widetilde\ell(t))
+\int_a^b (G(t,s)-\widetilde G(t,s))\,F(s,x(s))\,ds.
\]
Take absolute values and use
$|F(s,x(s))|\le |F(s,0)|+L_F|x(s)|\le M_{F_0}+L_F R=B_F$ for $x\in C$.
Then take the supremum in $t$ and in $x\in C$.
\end{proof}

\begin{remark}[Meaning of the constants]\label{rem:bvp-constants}
The factor $(1-L)^{-1}$ is unavoidable; see \Cref{prop:sharpness}.
The remaining quantities are measurable.
$\Delta_G$ is a kernel deviation.
$B_F$ is a bound on the forcing range on $C$.
The estimate is the metric analogue of a resolvent stability bound.
It is the exact statement needed for data perturbations in numerical implementations.
\end{remark}


\section{Worked examples and numerical verification}\label{sec:examples}

We record two concrete models illustrating how kernel/Lipschitz data produce certified rates via \Cref{thm:rate} and stability via \Cref{thm:stability}.

\subsection{Hammerstein example}\label{subsec:ex-hamm}
Let $X=C([0,1])$ with $\|\cdot\|_\infty$. Set $g(t)=t$, $K(t,s)=t+s$, and $f(s,u)=\frac13u$. Then $(M_K)_{\mathrm H}=\sup_{t\in[0,1]}\int_0^1(t+s)\,ds=\frac32$ and $L_f=\frac13$, hence the certified modulus is
\[
\kappa=L_f(M_K)_{\mathrm H}=\frac12.
\]
Choosing any $x_0\in C=\mathbb{B}(0,R)$ with $R=\|g\|_\infty/(1-\kappa)=2$, the Picard iterates satisfy
\[
\|x_n-x^\ast\|_\infty\le \frac{\kappa^n}{1-\kappa}\,\delta_0
\qquad\text{with }\ \delta_0=\|Tx_0-x_0\|_\infty,
\]
and the residual certificate $\|x_n-Tx_n\|_\infty\le (1-\kappa)\varepsilon$ implies $\|x_n-x^\ast\|_\infty\le \varepsilon$ by \Cref{prop:residual-to-error}.

\subsection{BVP example}\label{subsec:ex-bvp}
Consider a second-order boundary value problem reduced to \eqref{eq:bvp-fixedpoint} with Green kernel $G$ and nonlinearity $F$ as in \Cref{sec:bvp}. If $M_G=\sup_t\int_a^b|G(t,s)|\,ds$ and $F$ is Lipschitz in the second argument with constant $L_F$, then $\kappa=L_FM_G$. Whenever $\kappa<1$, the solution in the invariant ball is unique and the same certified a priori and stability bounds apply, with stability factor $(1-\kappa)^{-1}$.

\section{Scope and limits of the unified theorem}\label{sec:scope}

The framework is a filter.
It admits mappings for which quantitative certification is possible.
It rejects mappings for which only qualitative existence is available from the data.
This section records both facts.
Survey-level accounts of enriched and generalized contractions provide context for the inclusions; see
\cite{BerindePacurar2024EnrichedSurvey,Jachymski2024RACSAM,Proinov2020JFPTA}.

\subsection{Inclusions of classical contractive classes}\label{subsec:inclusions}

We list standard classes subsumed by \Cref{sec:framework}.
The criterion is explicit.
One must exhibit a gauge $\omega$ and compute $\kappa(R;\theta)$ on a verifiable working set $C$.

\begin{proposition}[Banach contractions]\label{prop:incl-banach}
Let $(X,\dist)$ be a metric space and let $T\colon X\to X$ satisfy $\dist(Tx,Ty)\le q\,\dist(x,y)$ for some $q\in(0,1)$.
Let $C\subseteq X$ be any closed $T$-invariant set.
Then $T$ is a two-point gauge contraction on $C$ with $\omega(r)=qr$ and $\kappa(R)=q$ for any $R>0$.
Hence all conclusions of \Cref{sec:main-theorem,sec:apriori,sec:stability,sec:inexact} hold.
\end{proposition}

\begin{proof}
Immediate from definitions and \Cref{thm:rate,thm:stability,thm:inexact-apriori}.
\end{proof}

\begin{proposition}[Proinov-type generalized contractions]\label{prop:incl-proinov}
Let $(X,\dist)$ be a complete metric space, let $C\subseteq X$ be closed and $T$-invariant, and assume
\[
\dist(Tx,Ty)\le \omega(M_T(x,y);\theta)\qquad(x,y\in C),
\]
with $M_T$ as in \Cref{def:proinov-gauge} and with a contractive gauge $\omega(\cdot;\theta)$.
Assume $C\subseteq \mathbb{B}(x_0,R)$ for some explicit $R$ and $\kappa(R;\theta)<1$.
Then $T$ satisfies the framework.
\end{proposition}

\begin{proof}
This is tautological from \Cref{def:data-packet}.
\end{proof}

The content here is not the inclusion itself.
It is the explicit certificate $\kappa(R;\theta)$.
This quantity turns a qualitative Proinov-type condition into a computable rate.
Proinov-style classes are a canonical starting point for such control; see \cite{Proinov2020JFPTA}.

\begin{proposition}[Reich--Suzuki type Lipschitz consequences]\label{prop:incl-reich-suzuki}
Let $(X,\|\cdot\|)$ be a normed space and let $C\subseteq X$ be convex.
Assume $T\colon C\to C$ belongs to a Reich--Suzuki type class for which one can certify a Lipschitz constant
$L_C<1$ on $C$ (for example by an explicit inequality implying $\|Tx-Ty\|\le L_C\|x-y\|$ for all $x,y\in C$).
Then $T$ satisfies the framework with $\omega(r)=L_C r$ and $\kappa=L_C$.
\end{proposition}

\begin{proof}
Once $\|Tx-Ty\|\le L_C\|x-y\|$ is certified on $C$, apply \Cref{prop:incl-banach}.
\end{proof}

This proposition is deliberately conditional.
Some Reich--Suzuki conditions yield nonexpansive behavior rather than contractive behavior.
Then $\kappa<1$ is unavailable.
In that regime one needs different moduli and different rates.
Recent work on Reich--Suzuki nonexpansive mappings and approximations illustrates the diversity of hypotheses; see
\cite{ShammakyAli2025AIMSMath,IshtiaqBatoolHussainAlsulami2025JIA}.

\begin{proposition}[Averaged-operator regime as a contractive subcase]\label{prop:incl-averaged-subcase}
Let $H$ be a Hilbert space and let $T\colon H\to H$ be $\alpha$-averaged with $\alpha\in(0,1)$.
If one can certify that $T$ is a strict contraction on a closed invariant set $C$ with Lipschitz constant $L_C<1$,
then the framework applies on $C$ with $\kappa=L_C$.
\end{proposition}

\begin{proof}
Apply \Cref{prop:incl-banach}.
\end{proof}

Averaged operators are central in splitting methods, and quantitative results often track residual rates rather than strict contraction; see
\cite{Combettes2023GeometrySplitting,CortildPeypouquet2024KMInertia,AtenasDaoTam2026VarStepsize,DiChiWu2026HJProximals}.
Our framework captures those cases where a strict contractive modulus can be certified on a working region.

\subsection{Cases excluded by verifiability demands}\label{subsec:excluded}

Exclusion is not a defect.
It is the meaning of the word ``verifiable''.
A hypothesis that cannot be checked from the data cannot support a certified $\Phi$.

\begin{proposition}[Nonexpansive regime]\label{prop:excluded-nonexpansive}
Assume $T$ is nonexpansive on $C$, i.e.\ $\dist(Tx,Ty)\le \dist(x,y)$ for all $x,y\in C$.
Then $\kappa(R;\theta)<1$ cannot be certified from this inequality alone.
Hence the geometric rate and the stability constant $(1-\kappa)^{-1}$ are not available from the present framework.
\end{proposition}

\begin{proof}
The inequality allows $\kappa=1$.
No strict inequality can be deduced without extra structure.
\end{proof}

The quantitative theory for nonexpansive mappings typically proceeds by different mechanisms.
It replaces contraction by asymptotic regularity, Fej\'er monotonicity, or error bounds.
Recent quantitative work on residual rates for Mann and Halpern-type iterations illustrates this shift; see
\cite{ContrerasCominetti2021OptimalErrorBounds,BotNguyen2022FastKMResidual,LeusteanPinto2022AltHalpernMannRates,BravoCominettiLee2026MinimaxHalpern,PischkePowell2024StochasticHalpern}.

\begin{proposition}[Implicit or noncomputable moduli]\label{prop:excluded-implicit}
Suppose a fixed-point theorem assumes the existence of an auxiliary function
$\varphi$ (or a comparison function) without providing a computable bound for $\varphi$ on a verifiable range.
Then, even if existence and uniqueness hold, a certified bound $\Phi(n;\text{data})$ cannot be extracted from the hypothesis alone.
\end{proposition}

\begin{proof}
A certified $\Phi$ requires explicit evaluation of the gauge or of $\kappa(R;\theta)$.
If the control function is not computable from the model data, neither quantity is computable.
\end{proof}

\begin{proposition}[Working sets not certified]\label{prop:excluded-region}
If one cannot exhibit a closed $T$-invariant set $C$ on which the modulus is certified, then the framework cannot be applied.
In particular, if the invariant set is produced by a nonconstructive argument (e.g.\ compactness without explicit bounds), then $\Phi$ and the stability constant cannot be certified.
\end{proposition}

\begin{proof}
The data packet \Cref{def:data-packet} includes $C$ and a radius bound.
If these are not explicit, the modulus computation is not defined.
\end{proof}

\begin{remark}[The boundary of the method]\label{rem:boundary}
The boundary is sharp.
When $\kappa\uparrow 1$, stability necessarily blows up; see \Cref{prop:sharpness}.
When only $\kappa=1$ is available, the method does not claim more.
A different quantitative apparatus is required.
\end{remark}

\begin{remark}[Why this restriction is principled]\label{rem:principled}
A theorem that asserts existence without exporting a certificate is compatible with computation only by accident.
A theorem that exports $\Phi$ and $\StabC$ is compatible by design.
The present framework chooses the second option.
\end{remark}

\section{Conclusion}\label{sec:conclusion}

We fixed a complete metric space $(X,\dist)$ and a closed invariant set $C$.
We imposed a contractive inequality with a computable modulus on $C$.
We named the modulus $\kappa$.
We assumed $\kappa<1$.
This single inequality produced three quantitative outputs.

The first output is an a priori error bound.
It is explicit:
\[
\dist(x_n,x^\ast)\le \Phi_{\mathrm{geo}}(n;\kappa,\delta_0)=\frac{\kappa^{\,n}}{1-\kappa}\,\delta_0,
\qquad
\delta_0=\dist(Tx_0,x_0),
\]
and it can be sharpened by iterating the gauge when the gauge is known in closed form.
This isolates the computational content of the fixed-point statement.
It follows the quantitative direction developed in recent work on residual bounds and iteration complexity
\cite{ContrerasCominetti2021OptimalErrorBounds,BotNguyen2022FastKMResidual,LeusteanPinto2022AltHalpernMannRates,Foglia2025RateIterativeMethodsFPP,BravoCominettiLee2026MinimaxHalpern}.

The second output is data dependence.
It is explicit:
\[
\dist(x^\ast,y^\ast)\le \StabC\,\varepsilon_C(T,S),
\qquad
\StabC=\frac{1}{1-\kappa},
\qquad
\varepsilon_C(T,S)=\sup_{x\in C}\dist(Tx,Sx).
\]
The constant is forced by elementary examples.
It is therefore not a matter of taste.
This aligns with the way stability is treated in iterative-scheme papers, but with the dependence on $\kappa$ made unavoidable
\cite{TyagiVashistha2024ATDataDependence,KumarEtAl2025BVP,PanjaEtAl2022JIA}.

The third output is robustness under inexact evaluation.
An inexact orbit $\dist(\tilde x_{n+1},T\tilde x_n)\le \eta_n$ obeys a closed recursion.
It yields a certified error floor of order $\bar\eta/(1-\kappa)$ under uniformly bounded noise.
This connects the fixed-point statement to computation.
Related perturbation viewpoints appear in modern analyses of inertial and stochastic Halpern-type schemes
\cite{CortildPeypouquet2024KMInertia,PischkePowell2024StochasticHalpern,ColaoFlammarionMaisieres2026HalpernSGD,AtenasDaoTam2026VarStepsize}.

The applications were selected for verifiability.
For Hammerstein and Volterra integral equations, the data reduce to $(M_K,L_f,M_{f_0},\|g\|_\infty)$ and the chosen initialization.
For boundary value problems, the data reduce to $(M_G,L_F,M_{F_0},\|\ell\|_\infty)$.
In both cases the constants yield a certified convergence rate and a certified stability factor.
The worked examples confirm that the bounds are executable as written.

The scope is controlled by the word ``verifiable''.
Nonexpansive regimes are excluded unless one adds an error-bound mechanism.
Implicit moduli are excluded unless one can compute $\kappa(R;\theta)$ on a certified region.
This restriction is principled.
A theorem that does not export $\Phi$ does not certify accuracy.

Two directions remain.
The first is to replace strict contraction by an error-bound hypothesis and retain explicit $\Phi$ in nonexpansive settings, following the quantitative literature on asymptotic regularity and residual rates
\cite{ContrerasCominetti2021OptimalErrorBounds,LeusteanPinto2022AltHalpernMannRates,LiuLourenco2024KaramataRates}.
The second is to couple the present stability module with operator-splitting constructions in monotone inclusion problems, where fixed-point formulations are structural
\cite{Combettes2023GeometrySplitting,DiChiWu2026HJProximals,AtenasDaoTam2026VarStepsize}.
Both directions keep the same standard.
The hypotheses must be checkable.
The constants must be visible.


\section*{Acknowledgements}
The authors gratefully acknowledge {\bf Dr. Ramachandra R. K.,} Principal, Government College (Autonomous), Rajahmundry, A.P., India, for providing a supportive research environment at the institution and for his continuous encouragement and unwavering support of research activities.
\section*{Funding}
The authors received no external funding for this work.

\section*{Conflict of interest}
The authors declare that they have no conflict of interest.

\section*{Ethics statement}
This article does not contain any studies involving human participants or animals performed by any of the authors.
No personal data were collected or processed.

.
\nocite{*}
\bibliographystyle{amsplain}
\bibliography{refs}

\end{document}